\documentclass{article}
\usepackage{amsfonts}
\usepackage{amsmath}
\usepackage{amssymb}
\newcounter{excnt}
\newtheorem{thm}{Theorem}
\newtheorem{df}[thm]{Definition}
\newtheorem{rem}[thm]{Remark}

\newtheorem{lem}[thm]{Lemma}

\newenvironment{proof}%
     {\noindent {\bf Proof:}}{\nopagebreak\begin{flushright}{\bf q.e.d.}\end{flushright}}
     {\stepcounter{excnt}
      \noindent {\bf Example \arabic{excnt}} \phantom{i}}{\vspace{0.4cm}}

\title{A Modified Coefficient Ideal for Use with the Strict Transform }
\author{Anne Fr\"uhbis--Kr\"uger\\
        Institut f. Alg. Geometrie, Leibniz Universt\"at Hannover, Germany}

\begin{document}

\maketitle

\abstract{
Two main algorithmic approaches are known for making Hironaka's 
proof of resolution of singularities in characteristic zero constructive.
Their main differences are the use of different notions of transforms during
the resolution process and the different use of exceptional divisors in 
the descent in ambient dimension. In this article, we focus on the 
first difference. Only the approach using the weak transform has
up to now been successfully used in implementations, because the other
one requires an explicit stratification by the Hilbert-Samuel function
at each step of the algorithm which is highly impractical due to the high
complexity of the computation of such a stratification. In this article,
a (hybrid-type) algorithmic approach is proposed which allows the use 
of the strict transform without the full impact of the complexity of the 
stratification by the Hilbert-Samuel function at each step of the 
desingularization process. This new approach is not intended to always be
superior to the previously implemented one, instead it has its strengths
precisely at the weak point of the other one and is thus a candidate
to be joined with it by an appropriate heuristic.}

\section{Introduction}

Existence and construction of a desingularisation has been one of the 
central questions in algebraic geometry since the end of the 19th century.
In characteristic zero, it was proved in Hironaka's groundbraking work
\cite{Hir} in 1964, in which he also introduced standard bases w.r.t.
a local ordering among other tools, whereas the case of positive 
characteristic is still open. Nevertheless, the interest in resolution 
of singularities in characteristic zero did not end at that time. 
Instead, the main interest only shifted toward the quest for a better, 
more constructive understanding of Hironaka's non-constructive proof
in which the key is the choice of appropriate centers for the blowing
ups which provide the desingularization. This development led to two 
main algorithmic approaches to the proof: one due to Bierstone and Milman 
(see e.g. \cite{BM}) using a stratification by the Hilbert-Samuel-function 
at the beginning of each choice-of-center step and the strict transform
as the corresponding notion of transform. The other one, due to Villamayor
with many contributions and simplifications by others over the years (see 
e.g. \cite{BEV}, \cite{EH}, \cite{Kol}), does not need the Hilbert-Samuel 
function in the choice-of-center step, but pays for it by using the weak 
transform which (compared to the strict transform) picks up extra components
lying inside the exceptional divisor at each blowing up.\footnote{This 
difference is not the only one between the two approaches. They also 
differ in their descent in dimension of the ambient space. For Villamayor's
approach we can basically use any local coordinate system at a given point
and use the usual notion of derivatives. Bierstone and Milman, on the other
hand, pay special attention to a choice of the local coordinates, 
taking into account the exceptional divisors passing through the given 
point, and use logarithmic derivatives to ensure that they can factor
out larger powers of the exceptional divisors after the descent in 
dimension.}   This latter approach has been used in the two currently 
existing implementations, mainly because its invariant is more accessible 
to practical calculations (see e.g. \cite{BSch} and \cite{FKP}). 
These implementations, on the other hand, have brought certain questions 
and even conjectures, which had been treated purely theoretically up to 
that point, into the reach of computer experiments: among these e.g. 
explicit computation of the topological zeta-function in the quest for 
a counter example to the monodromy conjecture, treatment of singularities 
in singular learning machines and in hidden Markov models in algebraic 
statistics or systematic study of multiplier ideals for cases beyond 
plane curves (see e.g. \cite{W-Y}, \cite{Tuc}). 
These applications also showed that there are classes of examples, where 
the current implementations are rather far from choosing the optimal one 
among different possible sequences of centers. Additionally, the
orders/degrees of the higher order generators of the ideal in question
tend to grow much faster for the weak transform than for the strict transform.
As the order, Villamayor's main invariant, is blind to the order of
higher order generators, this has positive influence on the algorithmic proof,
but very negative one on the efficiency of the implemented algorithm.\\
Therefore it is a legitimate question to ask why the other approach has
not been implemented up to now. Here the stratification by the 
Hilbert-Samuel function or more precisely the task of finding the 
stratum of maximal Hilbert-Samuel function turned out to be the crucial 
issue. Computer experiments (see e.g. \cite{Raschke}) showed that even
in very careful implementations this step is far from being sufficiently
efficient to be used in every choice-of-center step, because it involves a
parametric standard basis calculation with as many parameters as there 
are variables in the basering. Additionally, the internal differences in the
descent in dimension, which were already mentioned in a footnote above, 
turn out to be without significant impact on the overall efficiency, 
because the fewer blowing ups needed by Bierstone and Milman have to be 
paid for by the algorithmically more involved use of a special coordinate 
system and logarithmic derivatives. Thus a direct implementation of the 
Bierstone-Milman approach does not promise any improvement in speed or size
as compared to the implemented algorithms. If, however, the calculation
of the Hilbert-Samuel stratum can be avoided in a significant number of 
steps, this approach could again be an option. In this article, we follow
this line of thought and provide an outline of an algorithm which does
not compute the Hilbert-Samuel stratum in each step, but instead computes
auxilliary ideals which provide the information, whether the Hilbert-Samuel 
function of the original ideal dropped, in terms of their order.
The geometric idea behind this is to modify the original ideal by adding
extra components which only emphasize the contribution by certain generators
of the ideal without having negative impact on the descending induction 
on dimension of the ambient space, which is the key to Hironakas proof
and all constructive variants of it. \\
In the section \ref{defcoef}, we give the definition of our modified 
coefficient ideal and then apply it to the problem of resolution of 
singularities in the following section. It is important to observe at
this point, that the modified coefficient ideal is just another way of 
stating that we pass through an intermediate auxilliary ideal before
entering the first descent in dimension. Section 5 and 6 are then of a
more practical nature, the first one discussing the computational issues
of this new approach and the last one illustrating the construction on
two simple explicit examples.\\
As the ideas leading to this article developed in several steps 
over time and underwent more than one metamorphosis before taking the
shape in which they are presented here, I am indepted to many people 
for fruitful discussions, among them Bernard Teissier, Vincent Cossart,
Herwig Hauser, Edward Bierstone, Roushdi Bahloul, Gerhard Pfister, 
Rocio Blanco and Patrick Popescu-Pampu. I'd also like to thank the 
Institut de Math\'ematiques de Jussieu for their hospitality, at which 
I was a visitor when my first ideas in this direction evolved.

\section{Basic Definitions and Notations}

In order to fix notation, we would like to first state the basic defintions
and some selected properties of the invariant controlling the choice of 
centers. A section of just 1 or 2 pages can obviously not suffice to even 
give a brief introduction into the intricacies of resolution of singularities,
instead we would like to point to more thorough discussions in section 
4.2 of \cite{FK1} %Triest
from the practical point of view and in \cite{EH} embedded in a detailed
treatement of the resolution process.\\
In the general setting for these definitions, $W$ is a smooth equidimensional 
scheme over an algebraically closed field $K$ of characteristic zero and 
$X \subset W$ a subscheme thereof. For the purpose of quickly fixing 
notation in this section, we immediately focus on one affine chart $U$
with coordinate ring $R$ and denote the maximal ideal at $x \in U$ by
${\mathfrak m}_x$; $x_1,\dots,x_d$ are a local system of parameters of 
$R$ at $x$.\\
The order of an ideal $I = \langle g_1,\dots,g_r \rangle \subset R$ at a 
point $x \in U$ is defined as
           $$ord_x(I):={\rm max}\{m \in {\mathbb N} | \;\;
                   I \subset {\mathfrak m}_x^m \};$$
the locus of order $\geq 2$ of $I$ can be computed as the vanishing locus of
           $$\Delta(I)=\langle \{g_i| 1 \leq i \leq s\} \cup
                               \{\frac{\partial g_i}{\partial x_j} |
                               1 \leq i \leq r, 1 \leq j \leq d \} \rangle,$$
the locus of order $\geq c$ as the one of $\Delta^{c-1}(I)$.\\
The Hilbert-Samuel function of $R/I$ at $x$ is defined as
\begin{eqnarray*} 
HS_x: {\mathbb N} & \longrightarrow & {\mathbb N}\cr
                s & \longmapsto     & length(R/{\mathfrak m}_x^{s+1}),
\end{eqnarray*}
comparison is performed lexicographically. Locally at a point it can be
computed by determining a standard basis at this point w.r.t. a local
degree ordering and subsequent combinatorics on its leading ideal. 
Determining its locus of maximal value, however, is an expensive parametric 
Groebner basis calculation involving as many parameters as there are
variables.\\
Having fixed notation for these invariants, we can now state the general
structure of the invariant which controls the choice of center in the 
resolution process:
$$(ord\;\; {\rm or}\;\; HS, n; ord, n; ord, n;\dots)$$
where $n$ denotes a count of certain exceptional divisors which is not
going to have any impact on the considerations in this article. The
semicolon in the invariant denotes the key step in the invariant,
Hironaka's descent in dimension of the ambient space by means of 
a hypersurface of maximal contact. Hypersurfaces of maximal contact 
are defined by order 1 elements of $\Delta^{max_x ord_x(I)-1}(I)$, which
satisfy certain normal crossing conditions; choosing a hypersurface of
maximal contact can be interpreted as locally choosing a main variable. 
The order directly after the semicolon then denotes the order of an 
auxilliary ideal, the coefficient ideal\footnote{More precisely, the
order of the non-monomial part of the coefficient ideal is taken after  
decomposing into a product of ideals consisting of a monomial part, i.e. 
a product of powers of the exceptional divisors, and a non-monomial part.}, 
arising in the descent in dimension by suitably collecting the coefficients 
of the powers of the main variable. More precisely, let $Z=V(z)$ be a 
hypersurface of maximal contact for $I$ at $x$, then the coefficient ideal 
of $I$ w.r.t. $Z$ can be computed as
         $$Coeff_Z(I)=\sum_{k=0}^{ord_x(I)-1} I_k^{\frac{k!}{k-i}}$$
where $I_k$ is the ideal generated by all polynomials which appear as
coefficients of $z^k$ in some element of $I$.\\
Having stated the basic notions necessary to fix notation for discussing
algorithmic choice of center, we also need to briefly consider notions of
transforms under blowing ups: Let $\pi: \tilde{U} \longrightarrow U$ 
be a blowing up at a center $C$ with exceptional divisor $E$.
Then the total transform of $X \subset U$ (defined by the ideal $I_X$) 
under the blowing up $\pi$ is given by $I_X {\mathcal O}_{\tilde{U}}$. 
The strict and weak transform can then be computed as
$$I_{X,strict}=(I_X {\mathcal O}_{\tilde{U}} : I_E^{\infty}) \;\;\;
  {\rm and} \;\;\;
  I_{X,weak}=(I_X {\mathcal O}_{\tilde{U}} : I_E^b)$$
where $b$ is the largest integer such that 
$I_E^b \cdot (I_X {\mathcal O}_{\tilde{U}} : I_E^b)=
I_X {\mathcal O}_{\tilde{U}}$.

\section{Defining a Modified Coefficient Ideal}  \label{defcoef}

In this section, we define our modified coefficient ideal at a point
$w \in W$ in a very special situation first and subsequently study whether
we can always pass from the general case to this particular setting.
In the special situation, order reduction for the appropriately marked 
modified coefficient ideal is equivalent to a drop in the 
Hilbert-Samuel function of the original ideal under strict transform. \\

\noindent
To this end, let $y_1,\dots,y_n$ be a local system of parameters for 
${\mathcal O}_{W,w}$ and let 
${\mathcal I}_{X,w}=\langle f_1,\dots,f_k \rangle$ be subject to 
the following conditions:
\begin{enumerate}
\item[(1)] $f_1,\dots,f_k$ is a reduced standard basis of ${\mathcal I}_{X,w}$
           with respect to a local degree ordering such that
           $y_1 > y_2 > \dots > y_n$. We assume these to be numbered
           such that $d_1 \leq d_2 \leq \dots \leq d_k$ where 
           $d_i=ord_w(f_i)$.
\item[(2)] There are integers $1 \leq e_1 \leq \dots \leq e_k \leq n$ 
           such that for each $1 \leq i \leq k$ all $V(y_l)$, 
           $1 \leq l \leq e_i$, are hypersurfaces of maximal contact 
           for the ideal 
           $$J_i=\langle \underline{y}^{\alpha^{(1)}}f_1,\dots,
                         \underline{y}^{\alpha^{(i-1)}}f_{i-1},f_i,\dots,f_k
                 \mid \alpha^{(r)} \in I_{r,i} \forall 1 \leq r \leq i-1
                 \rangle $$
           where the numbers $e_i$ are maximal with this property and
           $I_{r,i}$ is the set of all multi-indices for which
           $\sum_{j=1}^{e_s} \alpha^{(r)}_j \geq max \{0,d_s - d_r\}$ for all 
           $1 \leq  s < i$ and $|\alpha^{(r)}|=d_i - d_r$.
\end{enumerate}
To be able to reference coefficients of each $f_s \in J_i$ w.r.t. monomials 
in the variables $y_1,\dots,y_{e_i}$ separately, we 
write $f_s=\sum_{\beta} a_{\beta}^{(s,i)} \underline{y_{(i)}}^{\beta}$.

\begin{df}
The modified coefficient ideal of ${\mathcal I}_{X,w}$ is then defined
as the usual coefficient ideal of the ideal $J_k$ with respect to
$Z=V(y_1,\dots,y_{e_k})$. More explicitly,
$${\it Coeff}^{new}_Z ({\mathcal I}_{X,w})=
    \sum_{j=0}^{d_k-1} I_j^{\frac{d_k!}{d_k-j}}$$
where $I_j=\langle a_{\beta}^{(i,k)} \mid 
                   \sum_{l=1}^{e_k} \beta_l \leq j - d_k + d_i \rangle$.
\end{df}

As will be discussed in detail in section \ref{resolve}, this modified
coefficient ideal is in no way intended to be used in all
descents in dimension of the ambient space in the computation of
the value of the invariant. Instead, it only replaces the usual coefficient 
ideal in the descent of highest ambient dimension, accommodating for the use 
of the strict transform. In all further descents, we use again the 
usual coefficient ideal, which for the time being is the one in Villamayor's
approach, although the approach could easily be changed to use the one of 
Bierstone and Milman. \\

\begin{rem}
The use of a standard basis in the above definition is not necessary for
the coefficient ideal itself: as all elements of a standard basis can be 
expressed as linear combinations of the original generators over the base 
ring, all contributions to the coefficient ideal already originate from 
these. The fact that we do not require identiy, but only $\leq$
in the condition concerning the $I_j$ again emphasizes this point of 
view. In the subsequent section, however, the standard basis property
will be used when considering the effects of blowing up on the coefficient
ideal, because this allows a significantly simpler treatment of the strict
transform of ${\mathcal I}_{X,w}$. 
\end{rem}

The crucial issue about the above definition is its possible dependence 
on a particular choice of the local system of parameters. This, however,
does not prevent its use in a resolution invariant as the following
lemma observes:

\begin{lem}
Different choices of the $y_1,\dots,y_n$ (subject to conditions (1) and (2))
affect neither the order of the modified coefficient ideal nor the 
orders of the subsequent (usual) coefficient ideals.
\end{lem}

This, obviously, implies that the order of the non-monomial part is also
unaffected by those different choices, as the order of the monomial part
is not sensitive to it.
 
\begin{proof}
First we observe that for $J_k$ itself different choices of the
hypersurfaces of maximal contact do not affect the orders of the
subsequent coefficient ideals, since we are using precisely the
usual Villamayor-style resolution invariant. What remains to be
proved is that the values of the invariant coincide for any two 
ideals $J_k$ and $J_k'$ arising from the original ideal 
${\mathcal I}_{X,w}$ as described above w.r.t. two different
local systems of parameters\footnote{These two systems of parameters
are of course both subject to the conditions (1) and (2).}
 $x_1,\dots,x_n$ and $y_1,\dots,y_n$. 
To this end, we first consider hypersurfaces of maximal contact for
$J_i$ which give rise to hypersurfaces of maximal contact for $J_{i+1}$ 
and then proceed by comparing the (usual) coefficient ideal of 
$J_k$ and $J_k'$ with respect to suitably chosen flags. \\[0.3cm]
\noindent \underline{Step 1:} descent in ambient dimension for $J_k$\\
Let $x_1,\dots,x_n$ be the regular system of parameters chosen in the 
construction of $J_k$ and let $V(y)$ be a hypersurface of maximal 
contact for $J_i$ for an arbitrary $1 \leq i < k$ (satisfying
$y \in \Delta^{d_i-1}(J_i)$ and $y \not\equiv 0 \mod {\mathfrak m_{W,w}}^2$). 
Recalling that $\Delta^{d_i-1}(J_i)$ is generated by the generators 
of $J_i$ and all their partial derivatives up to the $(d_1-1)$st ones,
we now consider the analogously constructed set of generators of 
$\Delta^{d_{i+1}-1}(J_{i+1})$. For those elements originating from standard
basis elements $f_s$ with $s>e_i$, we find again all generators we already had 
plus additionally higher derivatives thereof; for the elements originating
from an $f_s$, $s \leq e_i$, we now check the corresponding property
by explicit calculation:\\
Let $g=\frac{\partial^{\gamma} f_s}{\partial y^{\gamma}}$ be an 
arbitrary element of $\Delta^{d_i-1}(J_i)$ and let 
$\tilde{\gamma}=(\gamma_1+d_{i+1}-d_i,\gamma_2,\dots,\gamma_n)$.
Then 
$$\frac{\partial^{\tilde{\gamma}} y_1^{d_{i+1}-d_i} f_s}
       {\partial y^{\tilde{\gamma}}} =
\sum_{l=0}^{d_{i+1}-d_i} \left(\begin{array}{c} d_{i+1}-d_i \cr l
                        \end{array} \right) 
           \frac{\partial^l y_1^{d_{i+1}-d_i}}{\partial y_1^l} \cdot
           \frac{\partial^{(\tilde{\gamma}_1-l,\gamma_2,\dots,\gamma_n)} f_s}
                {\partial y^{(\tilde{\gamma}_1-l,\gamma_2,\dots,\gamma_n)}},$$
where the summand for $l=d_{i+1}-d_i$ is precisely the desired $g$ up
to a non-zero constant factor. But $y_1$ itself is an element of 
$\Delta^{d_i-1}(J_i)$ which implies that it can be written as a 
finite linear combination 
$$y_1 = \sum_{t} a_t \frac{\partial^{\eta_t} f_{i_t}}{\partial y^{\eta_t}},$$
where the $a_t$ are constants, $|\eta_t| \leq d_1-1$. As the appearing
derivatives of the $f_{i_t}$ for $i_t > e_i$ are in 
$\Delta^{d_{i+1}-1}(J_{i+1})$ by construction, we may assume without loss 
of generality that $i_t \leq e_i$ for all $t$.
Replacing the $f_{i_t}$ by 
$\frac{1}{(d_{i+1}-d_i)!} y_1^{d_{i+1}-d_i} \cdot f_{i_t}$ and the $\eta_t$
by $\tilde{\eta_t}$, the corresponding linear combination yields:
\begin{eqnarray*}
\sum_{t} a_t \frac{1}{(d_{i+1}-d_i)!}
             \frac{\partial^{\tilde{\eta_t}} y^{d_{i+1}-d_i} f_{i_t}}
                {\partial y^{\tilde{\eta_t}}} & = &
     \underbrace{\sum_{t} 
          a_t \frac{\partial^{\eta_t} f_{i_t}}{\partial y^{\eta_t}}}_{=y_1} 
     + y_1 \cdot c_2\\
 & = & (1+c_2) \cdot y_1,
\end{eqnarray*}
where $c_2$ is a positive constant.
Hence $y_1 \in \Delta^{d_{i+1}-1}(J_{i+1})$ which in turn implies by the
above considerations that indeed $g \in \Delta^{d_{i+1}-1}(J_{i+1})$ as
was to be proved.\\
This shows that indeed $\Delta^{d_i-1}(J_i)$ is contained in 
$\Delta^{d_{i+1}-1}(J_{i+1})$ which implies that 
$y_s \in \Delta^{d_{i+1}-1}(J_{i+1})$, $1 \leq s \leq e_i$ and hence 
proves the claim that $V(y_s)$ is a hypersurface of maximal contact for 
$J_{i+1}$. \\[0.3cm]
\noindent \underline{Step 2:} $y_i-x_i \in {\mathfrak m}_{W,w}^2$\\
Coming back to our original problem of comparing the coefficient ideals 
of $J_k$ and $J_k'$, we split our considerations into two parts. In this
first one, we assume that $g_l:=y_l-x_l \in {\mathfrak m}_{W,w}^2$, 
and use that $x_1,\dots,x_{e_k}$ give rise to 
hypersurfaces of maximal contact for $J_k'$ and $x_1-g_1,\dots,
x_{e_k}-g_{e_k}$ do so for $J_k$. A coordinate change replacing
$x_i$ by $x_i+g_i$ for all $1 \leq i \leq e_k$ transforms $J_k$ into 
$J_k'$ and $x_i-g_i$ into $x_i$, thus passing from one case to the
other. Therefore the appearing orders of all subsequent coefficient ideals
coincide.\\[0.3cm]
\noindent \underline{Step 3:} general case\\
By applying the construction of step 2 as a preparation step for one of 
the two systems of parameters and as a postprocessing step for the other
in the general case, we may restrict our considerations to a linear 
change of coordinates such that both systems of parameters satisfy 
the conditions (1) and (2). 
In this case, we do not even need to
worry about the orders, as the coefficients of the monomials of a given
degree are only transformed into linear combinations thereof still in the
same degree (according to appropriate matrices with constant entries having 
full rank). Hence the corresponding orders are unchanged.
\end{proof}

After studying the effects of different choices of the system of parameters
on the resulting resolution invariant, we now discuss how to pass from
the general situation to the special situation by a variant of the standard
basis algorithm: \\

Let ${\mathcal I}_{X,w}$ be as above and let $d_1:=ord_w({\mathcal I}_{X,w})$
be its order. Then we know that ${\mathcal I}_{X,w}$ contains 
at least one generator of order $d_1$ at $w$. We further know that
$e_1 := dim_K (\Delta^{d_1-1}({\mathcal I}_{X,w})/{\mathfrak m}_{W,w}^2) > 0$.
Thus we can choose $y_1,\dots,y_{e_1} \in {\mathfrak m}_{W,w}$ giving 
rise to a basis of this finite dimensional vector space and extend this
to some local system of parameters. Expressing the generators of 
${\mathcal I}_{X,w}$ w.r.t. the new system of parameters, we enter the 
standard basis calculation choosing a local degree lexicographical 
ordering on the set of monomials in the chosen system of parameters.
In the standard basis algorithm, we only treat s-polynomials of pairs 
with original leading monomials in degree $d_1$, postponing all calculations
in higher degrees. After appropriate renumbering and interreduction of 
the generators, we may assume that $f_1,\dots,f_i$ are precisely the 
standard basis elements of order $d_1= \dots =d_i$ and satisfy
$LM(f_1) > LM(f_2) > \dots > LM(f_i)$, where each of these leading monomials
is some product of powers of $y_1,\dots,y_{e_1}$.\\ 
We now set $e_j:=e_1$ for all $1 \leq j \leq i$. From now on, we proceed
degree by degree through the standard basis algorithm, adding new generators
$f_j$ to our evolving standard basis as necessary, defining the 
corresponding orders $d_j$ accordingly and setting $e_j:=e_1$, until a
leading monomial appears which is not a product of the $y_1,\dots,y_{e_1}$.
We do not yet add this polynomial $f_{s+1}$ to the evolving standard basis, 
because we first need to take care of adjusting our local system of parameters
in the current degree $d_l$.
To this end, we form a second ideal $J$ by dropping the elements of our
evolving standard basis from the set of generators of the ideal and
subsequently adding all products $\underline{y}^{\alpha} f_j$, where 
$f_j$ is one of the previously dropped elements of the evolving standard 
basis and $|\alpha|=d_l-d_j$. By construction, we now have
$e_{s+1}:=dim_K (\Delta^{d_1+|\alpha|-1}(J)/{\mathfrak m}_{W,w}^2) > e_s=e_1$.
As before, we choose $y_{e_1+1},\dots,y_{e_r}$ such that $y_1,\dots,y_{e_r}$
give rise to a basis of this vector space, and extend it to a local system
of parameters.  From here on, we proceed as before degree by degree
through the standard basis calculation extending the standard basis,
defining new $d_l$ and $e_l$ as needed and passing to a new local system 
of parameters according to the above construction when appropriate.  

\begin{rem}
In an affine setting, a reduced standard basis can be computed in family 
along each given connected component of a given stratum w.r.t. the 
Hilbert-Samuel function. This fact will be very useful for the computational
point of view, which we will consider later on in section \ref{compute}.
\end{rem}

\section{The Modified Coefficient Ideal and the Resolution Process} 
         \label{resolve}

From the construction introduced in the previous section, it is obvious
that a center determined in this way will always be contained in the
locus of maximal value of the Hilbert-Samuel function of ${\mathcal I}_{X}$. 
To study the use of the modified coefficient ideal in the resolution process,
we would like to compare the effects of a blowing up in such a center to
our original ${\mathcal I}_X$ and to the auxilliary ideals $J_k$.\\

\begin{lem}
The maximal value of the Hilbert-Samuel function of the strict transform
of $I_{X,w}$ is smaller than the original one at $w$ if and only if the 
maximal order of the weak transform of $J_k$ is strictly less than 
$d_k=ord_w(J_k)$.
\end{lem}

\begin{proof}
If the maximal value of the Hilbert-Samuel function of the strict transform
has decreased under the current blowing up, let us consider an arbitrary
point $w_1$ in the preimage of $w$ under the blowing up. There is at least 
one element, say $g$, of the reduced standard basis whose order has dropped, 
because the use of the reduced standard basis allows us to compute the strict 
transform by considering the strict transforms of the elements of the 
standard basis. But this implies that the order of the strict transform of
each generator $y^\alpha g$ of $J_k$ is strictly less than $d_k$ and hence 
the order of the weak transform of $J_k$ can no longer be $d_k$.\\
To prove the converse, let us now assume that the maximal order of the weak
transform of $J_k$ is no longer $d_k$ and let us fix an arbitrary point $w_1$ 
in the preimage of $w$ under the blowing up. There is at least one
generator of $J_k$, say $y^\alpha h$, whose strict transform
has order less than $d_k$ at $w_1$. A priori two situations may have occured:
$ord_{w_1}(h_{strict}) < ord_w(h)$, which directly implies a drop in the 
maximal value of the Hilbert-Samuel function under this blowing up, or
$ord_{w_1} ((y^\alpha)_{strict}) < |\alpha|$. In this second case, we 
can obviously find one $y_j$ whose strict transform has order zero. But
this $y_j$ was chosen as a local equation of a hypersurface of maximal 
contact for ${\mathcal I}_{X,w}$ and hence occurs as a factor of at least
one term of lowest order in at least one standard basis element, say $g$ of 
${\mathcal I}_{X,w}$. But this implies $ord_{w_1}(g_{strict}) < ord_{w}(g)$
which again corresponds to a decrease of the Hilbert-Samuel function 
as claimed.
\end{proof}

\begin{lem}
If the maximal value of the Hilbert-Samuel function is unchanged under 
blowing up,
i.e. if $HS_w({\mathcal I}_{X})=HS_{w_1}(({\mathcal I}_{X})_{strict})$, 
the weak transform of $J_k$ coincides with the newly constructed 
$\tilde{J_k}$ of the strict transform of ${\mathcal I}_{X,w}$.
\end{lem}

\begin{proof}
As the construction of the $J_k$ uses a reduced standard basis as its
starting point und the strict transform may be computed on the level of
the strict transforms of the elements of a reduced standard basis, we
immediately see that the $d_i$ and the $e_i$ in the construction are unchanged
and the hypersurfaces of maximal contact may be chosen to be the strict
transforms of the previous ones. By direct calculation of the weak transform
of $J_k$, we can now check that the two ideals indeed coincide.
\end{proof}

Using the previous lemmata, we can now state a modified resolution 
invariant, which implicitly uses the Hilbert-Samuel function as the 
first building block, but only requires the explicit computation 
of its maximal value upon each drop of maximal order of an auxilliary 
ideal:
$$(ord(J_k),n_1;ord({\it Coeff}^{new}),n_2;ord({\it Coeff}^V),n_3;\dots)$$
in contrast to a Villamayor-style invariant
$$(ord({\mathcal I}_{X,w}),n_1;ord({\it Coeff}^V),n_2;\dots)$$
or a Bierstone-Milman-style invariant
$$(HS({\mathcal I}_{X,w}),n_1;ord({\it Coeff}^{BM}),n_2;\dots),$$
where in each case the $n_i$ denote counting of certain exceptional
divisors and ${ \it Coeff}$ should be seen as a symbolic notation for a
coefficient ideal of Villamayor and Bierstone-Milman respectively.\\
From a theoretical point of view, the only advantage of this new
resolution invariant lies in the fact that we may now treat the
use of the strict transform in the framework of order reduction.
From the practical point of view, on the other hand, this modified
invariant allows us to exploit order reduction of higher order generators
of the ideal speeding up the resolution process in a large class of
examples without trading the fewer blowing ups for the huge computations
necessary to determine the Hilbert-Samuel stratum in each step. 

\section{Computational Aspects} \label{compute}

For explicit calculations, it is most convenient to pass to an affine 
open covering of our smooth equidimensional scheme $W$ and consider 
${\mathcal I}_{X}(U) \subset {\mathcal O}_W(U)$ on each of the affine
open sets. For simplicity of presentation, we will assume from now
on that ${\mathcal O}_W(U)=K[y_1,\dots,y_n]$.\footnote{The same construction
can also be carried out in the case of ${\mathcal O}_W(U)=K[y_1,\dots,y_n]/J$
for some ideal $J$. In that case, however, it may be necessary to pass
to yet another open cover of $U$ such that the local system of parameters
for $W$ at each point of the fixed smaller open set can be induced by the
same set of elements. Even given such a set inducing the local systems of 
parameters on the whole open set, the subsequent computations still
become by far more technical, obstructing the view to the heart of the
considerations.} To define the modified coefficient ideal, which allows 
the use of the strict instead of the weak transform, our construction 
will proceed by iteration of two steps: We first need to determine  
(the next variety in) a suitable flag in $W$ and (the next parameter in) 
a regular system of paramters $x_1,\dots,x_n$ for $W$ subordinate 
to this flag. With respect to a local degree ordering on the set of 
monomials $Mon(x_1,\dots,x_n)$, we determine (an intermediate result up
to the current degree of) a reduced standard basis of 
$f_1,\dots,f_k \in {\mathcal I}_{X}(U) := I \subset K[x_1,\dots,x_n]$.
This standard basis provides the complete information about the 
Hilbert-Samuel function through combinatorial reasoning on the staircase;
the coefficients of the $f_i$ w.r.t. to an appropriate subset of
$\{x_1,\dots,x_i\} \subset \{x_1,\dots,x_n\}$ will subsequently be 
used to compute the modified coefficient ideal.\\

More precisely, the construction of the flag and the standard basis
can be stated as follows:\\

\noindent
\underline{Step0}: Initialization\\

Let $d_1$ be the order of the ideal $I$. Create a copy $I_{orig}$ of
$I$ for later use. \\

\noindent
\underline{Step1}: Find appropriate hypersurfaces\\ 

As the order of $I$ is $d_1$, the order of $\Delta^{d_1-1}(I)$
is one and any order-1-element thereof (subject to the appropriate normal
crossing conditions) provides a hypersurface of maximal
contact for $I$. If necessary cover the affine chart by finitely many open 
sets $V_i$ such that on each of these we may use the same hypersurface at all
points of the locus of maximal order $V(\Delta^{d_1-1}(I))$. This hypersurface
gives rise to our first coordinate $y_1$. Similarly\footnote{To this end, 
reduce $\Delta^{d_1-1}(I)$ w.r.t. the new $y_1$ and subsequently consider
its order again.}, we can proceed to determine $y_2,\dots,y_{e_1}$, if 
there are further generators of order 1 of $\Delta^{d_1-1}(I)$. \\
As we will use these new $y_i$ as the first elements of the set inducing
at each point of the Hilbert-Samuel stratum a local regular system of 
parameters, it is convenient to use them as new variables replacing
existent ones, be it directly, by passing to an open covering, by 
finite extension of our ground field or an appropriate combination of
these methods.\\

\noindent
\underline{Step2}: do SB in 'degree' $d_1$\\

Now we express (at least) the order $d_1$ terms of the generators of 
$I$ in terms of these new $y_i$. We then proceed through the standard
basis algorithm in this degree $d_1$ in the sense that we form spolys of
all pairs arising from generators in this degree and reduce them w.r.t.
these generators until they are either reduced or themselves of order 
at least $d_1+1$.\\

\noindent
\underline{Step3}: proceed to next degree\\
We note for all generators of $I_{orig}$ which appeared as degree $d_1$ 
generators of $I$ that the respective degree is $d_1$.
Then we replace all order $d_1$ generators of $I$ by all possible products
of one such generator with one of the already chosen $y_j$. Hence the new 
ideal $I$ arising in this way has order $d_1+1$, with which we can return 
to steps 1 and 2 with the following modifications: the previously chosen 
$y_j$ stay unchanged and the next $d_k$ is defined, if additional $y_j$ 
arise -- in addition to newly formed spolys also older ones are reduced, 
whereever necessary. Step 3 can then be applied using the new $d_j$ and 
we return to step 1 again unless we have just marked the last generator 
of $I_{orig}$. \\
If all generators of $I_{orig}$ are marked, then we have found all 
contributions to the modified coefficient ideal, because all further
standard basis element only provide coefficients which are combinations of
coefficients already provided by lower order generators of $I_{orig}$. This 
is sufficient for our purposes and we can hence stop at this point.
(Note that this is not the complete standard basis computation and 
hence we cannot detect the whole Hilbert-Samuel function from it, but only
the first entries up to degree $d_k$.) \\

\noindent
\underline{Step4}: form modified coefficient ideal 
                   (cf. section \ref{defcoef})\\
On each chart which arose in the construction, we have now determined $e_k$ 
new smooth hypersurfaces $V(y_1),\dots,V(y_{e_k})$ and add $n-e_k$ further 
ones such that it gives rise to a regular system of parameters at each point 
of $V(\langle y_1,\dots,y_{e_k}\rangle)$. With respect to these $y_j$ we
can then determine the modified coefficient ideal.\\
Note that for this coefficient ideal, we can then proceed as in the original
algorithm of Villamayor.

\section{Examples}

The first example is mostly intended to illustrate what the respective
ideals look like and how they are transformed.\\
To see that the modified approach can really contribute to an improvement
of the performance of the resolution algorithm for certain classes of ideals, 
we subsequently state a very simple explicit example for which we have an 
immediate improvement of the maximal value of the Hilbert-Samuel function 
of the strict transform, whereas the maximal order of the weak transform 
stays $2$.

\subsection{$V(z^2+x^3y^3,w^5+x^5+v^3y^2) \subset {\mathbb A}_{\mathbb C}^5$}

In this case the ideal to be considered is
$$I=\langle z^2+x^3y^3,w^5+x^5+v^3y^2 \rangle$$
which is already a standard basis w.r.t. a local degree reverse lexicographical
ordering on $Mon(x,y,z,w,v)$.\\
The first auxilliary ideal 
$\Delta(I)=\langle z,x^2y^3, x^3y^2, w^4, x^4, v^3y, v^2y^2 \rangle$ is
obviously of order $1$ and 
$dim_{\mathbb C} (\Delta(I)/\langle x,y,z,w,v \rangle^2) =1$. We thus choose
the first hypersurface of maximal contact to be defined by $y_1:=z$.
As this is already one of our coordinates, we do not need any coordinate
change at this point.\\
Following through the algorithmic steps of section 4, the next interesting
degree is 5, where our corresponding ideal $J$ has the structure:
$$J=\langle z^5+z^3x^3y^3, w^5+x^5+v^3y^2 \rangle,$$
which is of order $5$ and allows the hypersurfaces of maximal contact
$V(x)$, $V(y)$, $V(z)$, $V(w)$, $V(y)$ already implying that the
upcoming center should be the origin.\\
We now consider the situation after the blowing up at the origin in the
various charts which we label (for convenience of the reader) by the
generator of the exceptional divisor on this chart:
\begin{description}
\item[Chart 1:] $E=V(x)$\\
                $I_{strict}=\langle z^2+x^4y^3, w^5+1+v^3y^2 \rangle$
                which can easily be checked to allow at most order $1$ at
                all points of this chart. In particular, the order and 
                hence the Hilbert-Samuel function have decreased under
                this blowing up. The upcoming center will be determined
                inside the hypersurface $V(w^5+1+v^3y^2)$.
                (The {weak} transform of $J$ is
                $\langle z^5+z^3x^4y^3, w^5+1+v^3y^2 \rangle$ which
                shows that its order has also dropped as was to be
                expected.)\\
                For comparison, we now observe that
                $I_{weak}=\langle z^2+x^4y^3, x^3w^5+x^3+x^3v^3y^2 \rangle$
                where the maximal order is still $2$ and the order of 
                the second generator has not decreased as much as
                in the case of the strict transform. Leaving us with
                the first hypersurface of maximal contact being $V(z)$
                and a coefficient ideal of order $3$ arising in this step.
\item[Chart 2:] $E=V(y)$\\
                $I_{strict}=\langle z^2+x^3y^4, w^5+x^5+v^3 \rangle$
                which is still of order $2$ at the zero locus of the
                ideal $\Delta(I)=\langle z,x^2y^4,x^3y^3,w^4,x^4,v^2 \rangle$
                that is along the line $V(x,z,w,v)$. But the drop of 
                order of the second generator causes a decrease of the 
                maximal value of the Hilbert-Samuel, which is now
                $(1,5,14,29,\dots)$ along the line $V(x,z,v,w)$
                as compared to $(1,5,14,30,\dots)$ at the previous origin.
                The weak transform of $J$ is
                $\langle z^5+z^3x^3y^4, w^5+x^5+v^3\rangle$ which
                shows that its order has also dropped as was to be
                expected.\\
                For comparison, we observe here that
                $I_{weak}=\langle z^2+x^3y^4, y^3w^5+y^3x^5+v^3y^4\rangle$
                where the maximal order is still $2$, but the
                order of the second generator has even increased
                producing a coefficient ideal of order $7$ (before 
                splitting into monomial part $y^3$ and a non-monomial part
                of order $4$).
\item[Chart 3:] $E=V(z)$\\
                $I_{strict}=\langle 1+x^3y^3z^4, w^5+x^5+v^3y^2 \rangle$
                of order at most $1$ (analogous to the first chart).
                $I_{weak}=\langle 1+x^3y^3z^4, w^5z^2+x^5z^2+v^3y^2z^2\rangle$
                also of order $1$.
\item[Chart 4:] $E=V(w)$\\
                $I_{strict}=\langle z^2+x^3y^3w^4, 1+x^5+v^3y^2\rangle$
                of order at most $1$ (again analogous to the first chart).
                $I_{weak}=\langle z^2+x^3y^3w^4, w^3+x^5w^3+v^3y^2w^3 \rangle$
                of order $2$ giving rise to a coefficient ideal of order
                $3$.
\item[Chart 5]  $E=V(v)$\\
                $I_{strict}=\langle z^2+x^3y^3v^4, w^5+x^5+y^2 \rangle$
                of order $2$, with a drop in the maximal value of the
                Hilbert-Samuel function to $(1,5,13,25,\dots)$ (analogous
                to the second chart).
                $I_{weak}=\langle z^2+x^3y^3v^4, w^5v^3+x^5v^3+y^2v^3 \rangle$
                of order $2$ giving rise to a coefficient ideal of order
                $5$ (before splitting into a monomial part $v^3$ and 
                a non-monomial part of order $2$).
\end{description}
In this example, the order of the ideal was only two, to keep all
computations at a level of complexity which can still be followed
without difficulty. This also made sure that forming the usual coefficient
ideals of the weak transforms the highest powers of ideals which needed
to be computed were second powers.\\ 
If, however, the order of the original ideal is higher, the appearance 
of a factorial of the previous order in the exponents easily leads to 
far higher powers in the construction of the coefficient ideal which is
of course iterated several times in Villamayor's approach. In our
proposed approach we make sure that we descend the maximal possible 
number of hypersurfaces of maximal order at the very first descend
of ambient dimension, thus minimizing the effect of taking iterated 
factorials. 

\subsection{$V(x^5+y^{11},z^9+x^9) \subset {\mathbb A}^3$}

As monomial ordering we choose a negative degree reverse lexicographical 
ordering with $z>y>x$ implying that the given generators already form a
standard basis. Now the obvious choice of center is $V(x,y,z)$. 
Considering the strict and weak transforms as before, we now focus on 
the charts in which the exceptional divisors are $E=V(y)$ and $E=V(z)$
respectively, omitting the discussion of the third chart:
\begin{description}
\item[Chart 1:] $E=V(z)$\\
$I_{strict}=\langle x^5+y^{11}z^6,1+x^9 \rangle$ is already non-singular,
$I_{weak}=\langle x^5+y^{11}z^6,z^4+x^9z^4 \rangle$ leads to a coefficient
ideal of order $8$ (before splitting into monomial and non-monomial part)
requiring up to $8$-th powers of certain ideals in the computation.
\item[Chart 2:] $E=V(y)$\\
$I_{strict}=\langle x^5+y^6,z^9+x^9 \rangle$ leads by our algorithm to
a locus of maximal Hilbert-Samuel function $V(x,y,z)$ and will be resolved
after the subsequent blowing up;
$I_{weak}=\langle x^5+y^6,y^4z^9+y^4x^9 \rangle$ gives rise to a coefficient
ideal of order $10$ requiring up to 60-th powers of certain ideals in the 
computation.
\end{description}

\begin{rem}
As can be seen from the two previous examples, improvements of the 
choice of centers arise from the new variant of the algorithm, 
whenever the lowest order generator of the given ideal is harder to 
resolve than another one, which is of significantly higher order, 
but 'simpler' structure. Thus a heuristic choosing between the two
approaches could take into account:
\begin{itemize}
\item number of generators of the ideal
\item degrees of generators of the ideal
\item number of terms in lowest order generators.
\end{itemize}
In the presence of two CPUs, another way of joining the two approaches 
could be to start them parallely each on one CPU and interrupt the 
slower one as soon as the faster one returned a result. These lines
of thought, however, are beyond the scope of this short article and
have not been explored systematically up to now. 
\end{rem}

\end{document}